\documentclass[preprint,runningheads]{svjour3}

\let\vec\relax
\DeclareMathAccent{\vec}{\mathord}{letters}{"7E}

\usepackage{amsmath}
\usepackage{amsfonts}
\usepackage{amssymb}
\usepackage{mathtools}
\usepackage{mathrsfs,bbm}
\usepackage{xcolor}
\usepackage{lmodern}

\usepackage{graphicx}
\usepackage{epstopdf}

\usepackage{bm}

\usepackage[ruled,vlined,linesnumbered]{algorithm2e}

\usepackage{accents}

\journalname{}
\date{ \phantom{b} \vspace{45mm}\phantom{e}}
\def\makeheadbox{\hspace{-4mm}Version of 3 April 2021}

\newcommand*{\dt}[1]{ \accentset{\mbox{\large\bfseries .}}{#1}}
\DeclareMathOperator{\bigtimes}{{\hbox{\large\sf X}}}

\def\iu{{\rm i}}

\def\R{{\mathbb R}}
\def\C{{\mathbb C}}
\def\eps{\varepsilon}

\def\wh{\widehat}
\def\wt{\widetilde}

\def\P{\mathrm{P}}

\newcommand\calM{\mathcal{M}}

\newcommand\bfr{{\mathbf r}}

\newcommand\bfI{{\mathbf I}}
\newcommand\bfA{{\mathbf A}}
\newcommand\bfB{{\mathbf B}}

\newcommand\bfD{{\mathbf D}}

\newcommand\bfF{{\mathbf F}}
\newcommand\bfG{{\mathbf G}}
\newcommand\bfH{{\mathbf H}}
\newcommand\bfK{{\mathbf K}}

\newcommand\bfM{{\mathbf M}}
\newcommand\bfN{{\mathbf N}}
\newcommand\bfP{{\mathbf P}}
\newcommand\bfQ{{\mathbf Q}}

\newcommand\bfS{{\mathbf S}}
\newcommand\bfU{{\mathbf U}}
\newcommand\bfV{{\mathbf V}}

\newcommand\bfY{{\mathbf Y}}
\newcommand\bfZ{{\mathbf Z}}

\newcommand\bfSigma{{\mathbf \Sigma}}

\newcommand\bfnabla{{\boldsymbol \nabla}}

\def\eps{\varepsilon}

\def\phi{\varphi}

\def\diag{\mbox{diag}}
\renewcommand{\Re}{{\mbox{\rm Re}}}

\author{Gianluca Ceruti, Jonas Kusch, and  Christian~Lubich}
\title{A rank-adaptive robust integrator for \\ dynamical low-rank approximation}
\date{}

\institute{G. Ceruti and Ch. Lubich \at
	Mathematisches Institut, Universit{\"a}t T{\"u}bingen, Auf der Morgenstelle 10, D-72076 T{\"u}bingen, Germany.
	\email{\{ceruti, lubich\}@na.uni-tuebingen.de}     
	\\
	J. Kusch \at Karlsruhe Institute of Technology, Kaiserstr. 12, 76131, Karlsruhe, Germany. \email{jonas.kusch@kit.edu}      
}

\begin{document}
	\maketitle
	
	\begin{abstract}  A rank-adaptive integrator  for the dynamical low-rank approximation of matrix and tensor differential equations is presented. The fixed-rank integrator recently proposed by two of the authors is extended to allow for an adaptive choice of the rank, using subspaces that are generated by the integrator itself. The integrator first updates the evolving bases and then does a Galerkin step in the subspace generated by both the new and old bases, which is followed by rank truncation to a given tolerance. It is shown that the adaptive low-rank integrator retains the exactness, robustness and symmetry-preserving properties of the previously proposed fixed-rank integrator. Beyond that, up to the truncation tolerance, the rank-adaptive integrator preserves the norm when the differential equation does, it preserves the energy for Schr\"odinger equations and Hamiltonian systems, and it preserves the monotonic decrease of the functional in gradient flows. Numerical experiments illustrate the behaviour of the rank-adaptive integrator.
	
%
%
%
%
	\keywords{dynamical low-rank approximation \and rank adaptivity \and structure-preserving integrator \and matrix and tensor differential equations}
	\subclass{65L05 \and 65L20 \and 65L70 \and 15A69}
\end{abstract}

	\section{Introduction}
	
	In \cite{CeL21}, a robust integrator for dynamical low-rank approximation of large matrix and tensor differential equations was proposed and analysed. As we show in this paper, that integrator allows for a remarkably simple extension to determine the rank adaptively, while retaining its favourable properties. In a step of the integrator of \cite{CeL21}, we first update the bases and then do a Galerkin step in the subspace generated by the {\it new} bases. In the rank-adaptive integrator proposed here, we do instead a Galerkin step in the larger subspace generated by both the {\it new and old} bases and then truncate to a given tolerance. This approach overcomes the problem of how to augment the bases to increase the rank, both in a very simple and very effective way, as will be demonstrated by our numerical experiments.
	
	Strategies of rank adaptivity for dynamical low-rank approximation in the context of the projector-splitting integrator of \cite{LubichOseledets,lubich2015time} have recently been proposed by Dektor, Rodgers \& Venturi \cite{DeRV20} and Yang \& White \cite{yang2020time}, and we are aware of ongoing work by Schrammer \cite{Schr21}. Those approaches are substantially different from what is proposed here.
	
	In Section 2 we present the rank-adaptive integrator for matrix differential equations and show that it has the same exactness property and robustness to small singular values as the fixed-rank integrator of \cite{CeL21}. It also preserves symmetry and skew-symmetry of the matrix when the matrix differential equation does. 
	
	In Section 3 we show further interesting features that are not available with the integrator of \cite{CeL21}. The following remarkable properties are satisfied in each step {\it up to the truncation tolerance} or a moderate multiple of it:
	\begin{itemize}
	\item The rank-adaptive integrator applied to a gradient system decreases the functional to be minimized.	
	\item The integrator preserves the norm when the differential equation does.
	\item The integrator applied to a matrix or tensor Schr\"odinger equation preserves the energy.
	\item The integrator applied to a Hamiltonian system (in an appropriate way) preserves the energy.
	\end{itemize}
These near-conservation properties come about by the Galerkin approach and the fact that the projection of the initial value to the augmented bases coincides with the  original initial value.
	
	In Section 4 we extend the rank-adaptive integrator to tensor differential equations whose solutions are approximated by Tucker tensors of varying multilinear rank. This is an extension of the fixed-rank tensor integrator of \cite{CeL21} that is analogous to the extension from fixed rank to adaptive rank in the matrix case.
	
	In Section 5 we present numerical experiments with examples from the fields of kinetic equations and uncertainty quantification, where dynamical low-rank approximation has recently found much interest, e.g. in \cite{EiHY21,EiJ21,EiL18,PeMF20} and \cite{FeL18,MuN18,MuNV20,SaL09}, beyond the original application area of quantum dynamics, e.g. \cite{MeMC90,MeyerGW09}  and \cite{haegeman2011time,haegeman2016unifying}.
	
	We describe the integrator for {\it real} matrices and tensors, but the algorithm and its properties extend in a straightforward way to {\it complex} matrices and tensors. This  only
requires care in using transposes $\bfU^\top$ versus adjoints $\bfU^*=\overline \bfU^\top$.
	
Throughout the paper,  matrices are written in boldface capital letters and tensors in italic capital letters.

	\section{A rank-adaptive robust low-rank matrix integrator}
	\label{sec:Par}
		Dynamical low-rank approximation of time-dependent matrices \cite{KochLubich07} approximates the solution $\bfA(t)\in\R^{m\times n}$ of a large (or too large) matrix differential equation
		\begin{equation} \label{ode-mat}
	\dt{\bfA}(t) = \bfF(t, \bfA(t)), 
	\qquad
	\bfA(t_0) = \bfA_0 
	\end{equation}
	by evolving matrices $\bfY(t)\in\R^{m\times n}$ of low rank, which are computed directly without first computing an approximation to the solution $\bfA(t)$. The initial low-rank matrix~$\bfY_0$ is typically obtained from a truncated singular value decomposition (SVD) of~$\bfA_0 $. Rank-$r$ matrices are represented in a non-unique factorized SVD-like form
	\begin{equation} \label{USV}
	\bfY = \bfU\bfS\bfV^\top, 
	\end{equation}
	where the slim matrices $\bfU\in \R^{m\times r}$ and $\bfV\in \R^{n\times r}$ each have $r$ orthonormal columns, 
	and the small matrix $\bfS\in \R^{r\times r}$ is invertible (but not necessarily diagonal).
	
		We present a modification of the fixed-rank integrator of \cite{CeL21} that retains its favourable properties but chooses the rank adaptively. This new integrator computes approximations $\bfY_n=\bfU_n\bfS_n\bfV_n^\top \approx \bfA(t_n)$ of an adaptively determined rank $r_n$ at discrete times $t_n$ ($n=0,1,2,\dots$). The stepsizes $h_n=t_{n+1}-t_n$ may also vary, but as we will not discuss stepsize selection in this paper, we take a constant stepsize $h>0$ for notational simplicity.
		
\subsection{Formulation of the algorithm} 
\label{subsec:alg-newint}
	One time step of integration from time $t_0$ to $t_1=t_0+h$  starting from a factored rank-$r_0$ matrix 
	$\bfY_0=\bfU_0\bfS_0\bfV_0^\top$ computes an updated factorization $\bfY_1=\bfU_1\bfS_1\bfV_1^\top$ of rank $r_1 \le 2r_0$.  In the following algorithm we let $r=r_0$ and we put a hat on quantities related to rank $2r$.
	
	\begin{enumerate}
		\item 
		Compute augmented basis matrices $\wh \bfU\in \R^{m\times 2r}$ and $\wh \bfV\in \R^{n\times 2r}$ (in parallel):
		\\[2mm]
		\textbf{K-step}:
		Integrate from $t=t_0$ to $t_1$ the $m \times r$ matrix differential equation
		\begin{equation}\label{K-step} 
		\dt{\textbf{K}}(t) = \bfF(t, \textbf{K}(t) \bfV_0^\top) \bfV_0, \qquad \textbf{K}(t_0) = \bfU_0 \bfS_0.
		\end{equation}
		Determine the columns of $\wh \bfU\in \R^{m\times 2r}$ as an orthonormal basis of the range 
		of the $m\times 2r$ matrix $(\textbf{K}(t_1),\bfU_0)$ (e.g.~by QR decomposition)
		and compute the $2r\times r$ matrix $\wh \bfM= \wh\bfU^\top \bfU_0$.
		\\[2mm]
		\textbf{L-step}: 
		Integrate from $t=t_0$ to $t_1$ the $n \times r$ matrix differential equation
		\begin{equation}\label{L-step} 
		\dt{\textbf{L}}(t) =\bfF(t, \bfU_0 \textbf{L}(t)^\top)^\top  \bfU_0, \qquad \textbf{L}(t_0) = \bfV_0 {\bfS}_0^\top. 
		\end{equation}
		Determine the columns of $\wh \bfV\in \R^{n\times 2r}$ as an orthonormal basis of the range 
		of the $n\times 2r$ matrix $(\textbf{L}(t_1),\bfV_0)$ (e.g.~by QR decomposition)
		and compute the $2r\times r$ matrix $\wh \bfN= \wh\bfV^\top \bfV_0$.
		\\[-2mm]
		\item
		Augment and update ${\bfS}_0 \rightarrow {\wh\bfS}(t_1)$\,: \\[1mm]
		\textbf{S-step}:  Integrate from $t=t_0$ to $t_1$ the $2r \times 2r$ matrix differential equation
		\begin{equation}\label{S-step} 
		\dt{\wh\bfS}(t) =  \wh\bfU^\top \bfF(t, \wh\bfU \wh\bfS(t) \wh\bfV^\top) \wh\bfV, 
		\qquad \wh\bfS(t_0) = \wh \bfM \bfS_0 \wh \bfN^\top.
		\end{equation}
		\item \textbf{Truncation}:
		Compute the SVD $\; \wh\bfS(t_1)= \wh \bfP \wh \bfSigma \wh \bfQ^\top$ with $\wh\bfSigma=\diag(\sigma_j)$ and truncate to the tolerance~$\vartheta$: Choose the new rank $r_1\le 2r$ as the 
		minimal number $r_1$ such that 
		$$
		\biggl(\ \sum_{j=r_1+1}^{2r} \sigma_j^2 \biggr)^{1/2} \le \vartheta.
		$$
		Compute the new factors for the approximation of $\bfY(t_1)$ as follows:
Let $\bfS_1$ be the $r_1\times r_1$ diagonal matrix with the $r_1$ largest singular values and let $\bfP_1\in \R^{2r\times r_1}$ and $\bfQ_1\in \R^{2r\times r_1}$ contain the first $r_1$ columns of $\wh \bfP$ and $\wh \bfQ$, respectively. Finally, set $\bfU_1 = \wh \bfU \bfP_1\in \R^{m\times r_1}$ and
$\bfV_1 = \wh \bfV \bfQ_1 \in \R^{n\times r_1}$.
	\end{enumerate} 
	The approximation after one time step is given by 
	\begin{equation}\label{Y1}
	 \bfY_1 = \bfU_1 \bfS_1 \bfV_1^\top \approx \bfY(t_1).
	 \end{equation}
	Then,  $\bfY_1$ is taken as the starting value for the next step, which computes $\bfY_2$ in factorized form, etc. 
	
The $m\times r$, $n\times r$ and $2r\times 2r$ matrix differential equations in the substeps are solved approximately using a standard integrator, e.g., an explicit or implicit Runge--Kutta method or an exponential integrator when $\bfF$ is predominantly linear.
%

The S-step is a Galerkin method for the differential equation \eqref{ode-mat} in the space of matrices $\wh\bfU \bfS \wh\bfV^\top$ generated by the extended basis matrices $\wh\bfU$ and $\wh\bfV$. 
Note that for $\bfY_0=\bfU_0\bfS_0\bfV_0^\top$, we have the projected starting value 
$\wh \bfU \wh \bfU^\top \bfY_0 \wh \bfV \wh \bfV^\top = \wh\bfU \wh \bfS(t_0) \wh \bfV^\top$.
The Galerkin step yields the rank-$2r$ approximation
\begin{equation}\label{whY1}
\wh \bfY_1 = \wh \bfU \wh \bfS(t_1) \wh \bfV^\top \approx \bfY(t_1).
\end{equation}

The above algorithm differs from the integrator in \cite{CeL21} in that the basis matrix $\wh \bfU$ not only contains an orthonormal basis of the range of $\bfK(t_1)$ but is extended to cover also the range of the initial basis $\bfU_0$, and this is done analogously for $\wh \bfV$. This extension allows us to increase the rank in a simple and effective way, while we retain the favourable properties of the integrator of  \cite{CeL21}.

\subsection{Exactness property and robust error bound}

	The new adaptive integrator shares very favourable properties with the low-rank matrix integrators of \cite{CeL21} and \cite{LubichOseledets}.
 First, it reproduces rank-$r$ matrices exactly.

	\begin{theorem}[Exactness]
		\label{thm:exact}
		Let $\bfA(t) \in \mathbb{R}^{m \times n}$ be of rank~$r$  for $t_0 \leq t \leq t_1$,
		so that $\bfA(t)$ has a factorization \eqref{USV}, $\bfA(t)=\bfU(t)\bfS(t)\bfV(t)^\top$. 
		Moreover, assume  that the $r\times r$ matrices  $\bfU(t_1)^\top \bfU(t_0)$ and $\bfV(t_1)^\top \bfV(t_0)$ are invertible. With $\bfY_0 = \bfA(t_0)$, the rank-adaptive integrator for $\bfF(t,\bfY)=\dt \bfA(t)$
		 is then exact: $ \bfY_1 = \bfA(t_1)$, provided that the truncation tolerance $\vartheta$ is smaller than the $r$-th singular value of $\bfA(t_1)$.
	\end{theorem}
	
	\begin{proof} We show that the arguments of the exactness proof of \cite[Theorem~3]{CeL21} apply with small modifications. As $\wh \bfS(t_0) = \wh\bfU^\top \bfU_0 \bfS_0 \bfV_0^\top \wh\bfV= \wh\bfU^\top \bfA(t_0) \wh\bfV$,
	we have by~\eqref{S-step}
	$$
	\wh \bfS(t_1) = \wh \bfS(t_0) + \int_{t_0}^{t_1} \wh\bfU^\top  \dt \bfA(t)  \wh\bfV \,dt
	= \wh \bfS(t_0) + \wh\bfU^\top \bigl( \bfA(t_1)-\bfA(t_0)\bigr) \wh\bfV =  \wh\bfU^\top \bfA(t_1) \wh\bfV.
	$$
	We observe that we can choose $\wh\bfU = (\wt\bfU_1,\wt\bfU_0)$, where $\wt\bfU_1$ is the orthogonal factor in the QR-decomposition of $\bfK(t_1)$, and $\wt\bfU_0^\top \wt\bfU_1=0$ by orthogonality. Lemma~1 of \cite{CeL21} shows that $\wt\bfU_1$ and $\bfA(t_1)$ have the same range, or equivalently,
	$$
	\wt\bfU_1 \wt\bfU_1^\top \bfA(t_1) = \bfA(t_1).
	$$
We then also have
         \begin{equation}\label{UUA}
	\wh\bfU \wh\bfU^\top \bfA(t_1) = \bfA(t_1),
	\end{equation}
	because the above equation together with $\wt\bfU_0^\top \wt\bfU_1=0$ implies
	\begin{align*}
	\wh\bfU \wh\bfU^\top \bfA(t_1) &=  \wt\bfU_1 \wt\bfU_1^\top \bfA(t_1) + \wt\bfU_0 \wt\bfU_0^\top  \bfA(t_1) 
	\\
	&=  \bfA(t_1) + \wt\bfU_0 (\wt\bfU_0^\top \wt\bfU_1) \wt\bfU_1^\top \bfA(t_1)  = \bfA(t_1).	
	\end{align*}
We note that \eqref{UUA} still holds true for a different choice of orthonormal basis $\wh \bfU$, since $\wh\bfU \wh\bfU^\top$ is the orthogonal projection onto the range of $(\bfK(t_1),\bfU_0)$, which does not depend on the particular choice of the orthonormal basis.
In the same way we obtain
         \begin{equation}\label{VVA}
	 \bfA(t_1)  \wh\bfV \wh\bfV^\top= \bfA(t_1).
	\end{equation}
By \eqref{UUA} and \eqref{VVA} we then have 	
        $$
	\wh\bfU\wh \bfS(t_1)\wh\bfV^\top=  \wh\bfU \wh\bfU^\top \bfA(t_1) \wh\bfV \wh\bfV^\top = \bfA(t_1).
	$$
As this matrix is of rank $r$, its truncation to rank $r$ leaves the result unchanged, and so we obtain the stated exactness result.
\qed
	\end{proof}

More importantly, the algorithm is robust to the presence of small singular values of the solution or its approximation, as opposed to standard integrators applied to 
the differential equations for the factors $\bfU(t)$, $\bfS(t)$, $\bfV(t)$, which contain a factor $\bfS(t)^{-1}$ on the right-hand sides \cite[Prop.\,2.1]{KochLubich07}. 
	The appearance of small singular values is to be expected in an adaptive low-rank approximation, because the smallest singular value retained in the approximation cannot be expected to be much larger than the small truncation tolerance $\vartheta$.
	The following error bound is independent of small singular values. Such a robust error bound was first shown in \cite[Theorem 2.1]{KieriLubichWalach} for the projector-splitting operator of \cite{LubichOseledets} and subsequently, based on that result, in \cite[Theorem~4]{CeL21} for the integrator of \cite{CeL21}.
		
	\begin{theorem}[Robust error bound]
		\label{thm:robust}
		Let $\bfA(t)$ denote the solution of the matrix differential equation \eqref{ode-mat}. Assume that  the following conditions hold in the Frobenius norm $\|\cdot\|=\|\cdot\|_F$:
		\begin{enumerate}
			\item 
			$\bfF$ is Lipschitz-continuous and bounded: for all $\bfY, \widetilde{\bfY} \in \mathbb{R}^{m \times n}$ and $0\le t \le T$,
			$$ 
			\| \bfF(t, \bfY) - \bfF(t, \widetilde{\bfY}) \| 
			\leq
			L \| \bfY - \widetilde{\bfY} \|,
			\qquad
			\| \bfF(t, \bfY) \| \leq B \ .
			$$
			
			\item
			The normal part of $\bfF(t, \bfY)$ is $\varepsilon$-small at rank $r_n$ for $\bfY$ near $\bfA(t_n)$ and~$t$ 
			near~$t_n$: With $\P_{r_n}(\bfY)$ denoting the orthogonal projection onto the tangent space of the manifold 
			$\calM_{r_n}$ of rank-$r_n$ matrices at $\bfY\in\calM_{r_n}$, it is assumed that
			$$
			\| (\bfI - \P_{r_n}(\bfY)) \bfF(t, \bfY) \| \le \eps
			$$
			for all $\bfY \in \mathcal{M}$ in a neighbourhood of $\bfA(t_n)$ and $t$ near $t_n$.
			
			\item
			The error in the initial value is $\delta$-small:
			$$
			\| \bfY_0 - \bfA_0 \| \le \delta.
			$$
		\end{enumerate}	
		Let $\bfY_n$ denote the low-rank approximation to $\bfA(t_n)$ at $t_n=nh$ obtained after n steps of the adaptive integrator with step-size $h>0$.
		Then, the error satisfies for all $n$ with $t_n =  nh \leq T$
		$$ \| \bfY_n - \bfA(t_n) \| \leq c_0\delta + c_1 \varepsilon + c_2 h + c_3 n\vartheta,$$	
		where the constants $c_i$ only depend on $L, B,$ and $T$. In particular, the constants are independent of singular values of the exact or approximate solution. 
	\end{theorem}
	
	The proof of this error bound follows the lines of the proof of \cite[Theorem~4]{CeL21} (with modifications of the same type as in the proof of Theorem~\ref{thm:exact}) and  is therefore omitted. The result is not fully satisfactory, as it does not show that the error improves when the truncation tolerance is made smaller, as is observed in numerical experiments. To show this, a proportionality relation between $\eps$ and $\vartheta$ would be needed, which is not available to us. Some of this effect becomes qualitatively plausible by noting that decreasing $\vartheta$ increases the rank, which decreases $\varepsilon$.
	
	 As in \cite[Section 2.6.3]{KieriLubichWalach},  an inexact solution of the matrix differential equations in the rank-adaptive integrator
	leads to an additional error that is bounded in terms of the local errors in the inexact substeps, again with constants that do not depend on small singular values.

 	\subsection{Symmetric and skew-symmetric low-rank matrices}
	We now assume that the right-hand side function in \eqref{ode-mat} is such that one of the following conditions holds,
	
	\begin{equation} \label{F-sym}
		\bfF(t,\bfY)^\top =  \bfF(t, \bfY^\top) \qquad \text{for all }\ \bfY \in \R^{n \times n}
	\end{equation}
	or
	\begin{equation} \label{F-skewsym}
	\bfF(t,\bfY)^\top =  -\bfF(t, -\bfY^\top) \qquad \text{for all }\ \bfY \in \R^{n \times n}.
	\end{equation}
	Under these conditions, solutions to \eqref{ode-mat} with symmetric or skew-symmetric initial data remain symmetric or skew-symmetric, respectively, for all times.
	
	\begin{theorem}
		Let $\bfY_0 = \bfU_0 \bfS_0 \bfU_0^\top \in \R^{n \times n}$ be  symmetric or skew-symmetric and assume that the function $\bfF$ satisfies property~$(\ref{F-sym})$ or $(\ref{F-skewsym})$, respectively.
		Then, the approximation $\bfY_1$ obtained after one time step of the new integrator is symmetric or skew-symmetric, respectively.
	\end{theorem}
	
	The proof is the same as in \cite[Theorem~5]{CeL21}.
	
\section{Structure preservation up to the truncation tolerance}
\label{sec:structure}
	
\subsection{Starting value of the Galerkin step}

The following relation, which is not satisfied for the integrator of \cite{CeL21}, will be essential in the following subsections. Here we use the notation of
\eqref{S-step}.

\begin{lemma}\label{lem:init}
Let  $\bfY_0=\bfU_0\bfS_0\bfV_0^\top$ and $\wh\bfY_0=\wh\bfU\wh\bfS(t_0)\bf\wh\bfV^\top$. Then, 
$\wh\bfY_0= \bfY_0$.
\end{lemma}

\begin{proof}
We note that by the definition of $\wh\bfS(t_0)$, we have
$
\wh\bfY_0  = \wh\bfU \wh\bfU^\top \bfU_0 \bfS_0 \bfV_0^\top \wh\bfV \wh\bfV^\top.
$
Here, $\wh\bfU \wh\bfU^\top$ is the orthogonal projection onto the range of $\wh\bfU$, which by definition equals the range of $(\bfK(t_1),\bfU_0)$. In particular, the columns of $\bfU_0$ are in the range of $\wh\bfU$, and hence $\wh\bfU \wh\bfU^\top\bfU_0=\bfU_0$.
In the same way we also have $\wh\bfV \wh\bfV^\top \bfV_0= \bfV_0$. So we obtain $\wh\bfY_0= \bfY_0$.
\qed
\end{proof}

\subsection{Norm preservation}
We now turn to a near-conservation property that is not satisfied with the integrator of \cite{CeL21}.
If the function $\bfF$ satisfies
\begin{equation}\label{F-cons}
\langle \bfY, \bfF(t,\bfY) \rangle = 0 \qquad\text{for all } \bfY \in \R^{m\times n} \text{ and all } t,
\end{equation}
then solutions of \eqref{ode-mat} preserve the Frobenius norm, i.e. $\| \bfA(t) \|= \| \bfA(0) \|$ for all $t$. For the proposed integrator, we show that each step preserves the norm up to the truncation tolerance $\vartheta$, independently of the stepsize $h$.

\begin{theorem}\label{thm:norm}
If $\bfF$ satisfies \eqref{F-cons}, then the numerical result $\bfY_1$ obtained after a step of the adaptive integrator with the truncation tolerance $\vartheta$ satisfies
$$
\bigl| \| \bfY_1 \| -  \| \bfY_0 \| \bigr| \le \vartheta.
$$
\end{theorem}

\begin{proof}
We show that the non-truncated result $\wh\bfY_1$ of \eqref{whY1} has the same norm as $\bfY_0$. The stated bound then follows from
$\bigl| \| \bfY_1 \| -  \| \wh\bfY_1 \| \bigr| \le \| \bfY_1 - \wh\bfY_1\| \le \vartheta$.
We begin by noting that $\|\wh \bfY_1\| = \| \wh\bfS(t_1) \|$ and $\|\bfY_0\| = \| \bfS_0 \|$, and further (omitting the argument $t$ after the first equality)
$$
\frac12 \,\frac{d}{dt} \, \| \wh \bfS(t) \|^2 = \langle \wh\bfS, \dt{\wh\bfS} \rangle = 
\langle \wh\bfS, \wh\bfU^\top \bfF(t, \wh\bfU \wh\bfS \wh\bfV^\top) \wh\bfV \rangle =
\langle \wh\bfU\wh\bfS\wh\bfV^\top\!,  \bfF(t, \wh\bfU \wh\bfS \wh\bfV^\top) \rangle = 0,
$$
where we used \eqref{F-cons} in the last equality. This yields $\| \wh\bfS(t_1)\| = \| \wh\bfS(t_0)\|$. Furthermore, we have 
$\|\wh\bfS(t_0)\| = \|\wh\bfY_0\|$, which by Lemma~\ref{lem:init} equals $\| \bfY_0\|$.


Altogether, we then have
$$
\|\wh \bfY_1\| = \| \wh\bfS(t_1) \|= \| \wh\bfS(t_0) \|=\|\wh\bfY_0\| = \|\bfY_0\|,
$$
which yields the result.
\qed
\end{proof}
%

\subsection{Gradient systems}

Consider a function $f:\R^{m\times n} \to \R$ that is to be minimized. Along every path $\bfA(t)$ of matrices, we have
$$
\frac{d}{dt} f(\bfA(t)) = \langle \bfG(\bfA(t)), \dt \bfA(t) \rangle,
$$
where $\langle\bfA,\bfB\rangle=\sum_{i,j} a_{ij}b_{ij}$ denotes the inner product that induces the Frobenius norm $\|\cdot\|=\|\cdot\|_F$,
and $\bfG(\bfA)=\bfnabla f(\bfA)\in\R^{m\times n}$ is the gradient. Clearly, $f$ decreases monotonically if $\bfA(t)$ is a solution of the gradient system
$$
\dt \bfA(t) = - \bfG(\bfA(t)).
$$
When we use the rank-adaptive dynamical low-rank integrator on the gradient system, then the function $f$ decreases along the low-rank approximations $\bfY_n$, up to terms of the order of the truncation tolerance $\vartheta$ times the gradient norm and errors made in the numerical integration of the $S$-step \eqref{S-step}. More precisely, we show the following.

\begin{theorem}\label{thm:grad}
The result $\bfY_1$ obtained after a step of the rank-adaptive integrator with the truncation tolerance $\vartheta$ applied to the gradient system for the function $f$ satisfies for some $\alpha,\beta\ge 0$
$$
f(\bfY_1)\le f(\bfY_0) - \alpha^2 h + \beta \vartheta.
$$
For $f$ with a Lipschitz gradient, we have $\alpha =  \| \wh\bfU^\top  \bfG(\bfY_0)  \wh\bfV \| + O(h) $ and 
$\beta= \| \bfG(\bfY_1) \| + O(\vartheta)$.
\end{theorem}

\begin{proof} Along the solution of the differential equation \eqref{S-step} with $\bfF(\bfY)=-\bfG(\bfY)$ we have, with $\wh\bfY(t) = \wh\bfU \wh \bfS(t) \wh\bfV^\top$,
\begin{align*}
\frac{d}{dt}\, f(\wh\bfY(t) ) &= \langle \bfG(\bfY(t)),  \wh\bfU \dt{\wh \bfS}(t) \wh\bfV^\top \rangle =
 \langle \wh\bfU^\top \bfG(\bfY(t))\wh\bfV, \dt{\wh\bfS}(t) \rangle 
 \\ &= \langle \wh\bfU^\top \bfG(\bfY(t))\wh\bfV, - \wh\bfU^\top \bfG(\bfY(t))\wh\bfV \rangle = -  \|  \wh\bfU^\top \bfG(\bfY(t))\wh\bfV \|^2  \le -\alpha^2
\end{align*}
with $\alpha = \min_{0\le\tau\le 1} \| \wh\bfU^\top  \bfG( \bfY(t_0+\tau h)  \wh\bfV \| $. Since $\wh\bfY(t_0)=\bfY_0$ 
by Lemma~\ref{lem:init} and $\wh\bfY_1=\wh\bfY(t_1)$, we obtain
$$
f(\wh\bfY_1) \le f(\wh\bfY_0) -\alpha^2 h.
$$
Since the truncation is such that $\|\bfY_1 - \wh\bfY_1\| \le \vartheta$, we have
$$
f(\bfY_1) \le f(\wh\bfY_1) + \beta \vartheta
$$
with
$\beta = \max_{0\le\tau\le 1} \| \bfG(\tau \bfY_1 + (1-\tau)\wh\bfY_1) \|$, and so we obtain the stated result.
\qed
\end{proof}

The above result does not include the error made in solving the differential equation \eqref{S-step} for $\bfS$ only approximately. If a step  with the implicit Euler method or a discrete gradient method is made, then $f$ still decreases along the numerical solution of this differential equation; cf. e.g.~\cite{HaL14}. Alternatively, a higher-order explicit method may give an accurate approximation to $\wh\bfS(t_1)$ and thus ensure a decrease in $f$.

\subsection{Schr\"odinger equations}
We now turn to the low-rank approximation of the  matrix Schr\"odinger equation
\begin{equation}\label{schroedinger}
\iu \dt \bfA(t) = \bfH[\bfA(t)].
\end{equation}
The Hamiltonian $\bfH:\C^{m\times n} \to \C^{m\times n}$ is a linear map that is self-adjoint, i.e.,
$$
\langle \bfH[\bfY], \bfZ \rangle = \langle \bfY, \bfH[\bfZ] \rangle \qquad\text{ for all } \bfY,\bfZ \in \C^{m\times n},
$$
where the complex inner product is given as $\langle \bfA,\bfB \rangle = \sum_{i,j} \overline a_{ij} b_{ij}$ so that the induced norm $\|\cdot\|$ is the Frobenius norm of complex matrices. The energy of a state (here: matrix) $\bfY$ of norm $1$ is 
$$
E(\bfY)= \langle \bfY, \bfH[\bfY] \rangle.
$$

We are now in a complex setting to which the real rank-adaptive integrator is readily extended: the transposes in the algorithm are replaced by  conjugate transposes. The result on norm preservation of the previous subsection applies also here, with essentially the same proof.
Remarkably, we also have energy preservation up to the order of the truncation tolerance.

\begin{theorem}\label{thm:schroedinger}
The numerical result $\bfY_1$ obtained after a step of the adaptive integrator with the truncation tolerance $\vartheta$ applied to the matrix Schr\"odinger equation \eqref{schroedinger} with $\bfY_0$ of norm $1$ satisfies
$$
\bigl| E(\bfY_1) - E(\bfY_0)  \bigr| \le \gamma\vartheta
$$
with $\gamma =  \|\bfH[\bfY_1] + \bfH[\wh\bfY_1] \|$.
\end{theorem}

\begin{proof} The proof is similar to that of Theorem~\ref{thm:grad}. Along the solution of the differential equation \eqref{S-step} with $\bfF(\bfY)=-\iu\bfH[\bfY]$ we have, with $\wh\bfY(t) = \wh\bfU \wh \bfS(t) \wh\bfV^*$,
\begin{align*}
\frac{d}{dt}\,  E( \wh\bfY(t))  &= 2\, \Re \langle \bfH[\wh\bfY(t)],  \wh\bfU \dt{\wh \bfS}(t) \wh\bfV^* \rangle 
= 2\, \Re \langle \wh\bfU^* \bfH[\wh\bfY(t)] \wh\bfV, \dt{\wh\bfS}(t) \rangle 
 \\ &= 2\, \Re \langle \wh\bfU^* \bfH[\wh\bfY(t)]\wh\bfV, - \iu\,\wh\bfU^* \bfH[\wh\bfY(t)]\wh\bfV \rangle
 \\ &= 2\, \Re\, (-\iu)\,\|\wh\bfU^* \bfH[\wh\bfY(t)]\wh\bfV\|^2 = 0.
 \end{align*}
Since $\wh\bfY(t_0)=\bfY_0$ 
by Lemma~\ref{lem:init} and $\wh\bfY_1=\wh\bfY(t_1)$, we obtain
$$
 E( \wh\bfY_1) = E(\bfY_0).
$$
Since the truncation is such that $\|\bfY_1 - \wh\bfY_1\| \le \vartheta$, we have by the Cauchy--Schwarz inequality
$$
|E(\bfY_1) - E(\wh\bfY_1)| =|\langle \bfY_1 -  \wh\bfY_1 , \bfH[\bfY_1 + \wh\bfY_1] \rangle|
\le \vartheta \,\|\bfH[\bfY_1] + \bfH[\wh\bfY_1] \|,
$$
which yields the stated result.
\qed
\end{proof}


\subsection{Hamiltonian systems}
Given a smooth Hamilton function $H:\R^{m\times n}\times \R^{m\times n} \to \R$, we consider the corresponding Hamiltonian differential equations
\begin{equation} \label{ham}
\dt \bfQ = \bfnabla_{\bfP} H(\bfQ,\bfP), \quad\ \dt \bfP = -\bfnabla_{\bfQ} H(\bfQ,\bfP).
\end{equation}
It is not advisable to do dynamical low-rank approximation in the usual way directly on these matrix differential equations. What we propose here, is to rewrite the differential equations \eqref{ham} in the complex variables
$$
\bfZ = \bfQ + \iu \bfP, \quad \overline \bfZ = \bfQ - \iu \bfP
$$
with the energy function $E$ defined by
$$
\tfrac12 E(\bfZ,\overline\bfZ)= H(\bfQ,\bfP),
$$
which yields the differential equation in Schr\"odinger form
\begin{equation} \label{eq:ham-schroed}
\iu \dt\bfZ = \bfnabla_{\bar  \bfZ} E(\bfZ,\overline\bfZ) \quad (= \bfnabla_\bfQ H + \iu \bfnabla_\bfP H).
\end{equation}
We then apply the complex version of the rank-adaptive integrator to this differential equation and finally separate real and imaginary parts to obtain approximations to $\bfQ(t),\bfP(t)$. With this approach, we obtain energy conservation up to a multiple of the truncation tolerance $\vartheta$, irrespective of the stepsize $h$.

\begin{theorem} The result $\bfZ_1=\bfQ_1+\iu\bfP_1$ obtained after a step of the rank-adaptive integrator with the truncation tolerance $\vartheta$ applied to the complex system \eqref{eq:ham-schroed} with initial value $\bfZ_0=\bfQ_0+\iu\bfP_0$ satisfies 
$$
\bigl| H(\bfQ_1,\bfP_1) - H(\bfQ_0,\bfP_0)\bigr| \le \beta \vartheta,
$$
where
$\beta=  \| \bfnabla H(\bfQ_1,\bfP_1) \| + O(\vartheta)$.
\end{theorem}

\begin{proof} The proof is similar to that of Theorems~\ref{thm:grad} and~\ref{thm:schroedinger}. Along the solution of the differential equation \eqref{S-step} with $\bfF(\bfZ)=-\iu \bfnabla_{\bar  \bfZ} E(\bfZ,\overline\bfZ)$ we have, with $\wh\bfZ(t) = \wh\bfU \wh \bfS(t) \wh\bfV^*$,
\begin{align*}
\frac{d}{dt}\, E(\wh\bfZ(t),\overline{\wh\bfZ(t)}) &= 2\,\Re\langle \bfnabla_{\bar  \bfZ} E(\wh\bfZ(t),\overline {\wh\bfZ(t)}),  \wh\bfU \dt{\wh \bfS}(t) \wh\bfV^* \rangle
\\ 
&= 2\,\Re\langle \wh\bfU^* \bfnabla_{\bar  \bfZ} E(\wh\bfZ(t),\overline {\wh\bfZ(t)})\wh\bfV, \dt{\wh \bfS}(t) \rangle 
 \\ &= 2\,\Re\langle \wh\bfU^* \bfnabla_{\bar  \bfZ} E(\wh\bfZ(t),\overline {\wh\bfZ(t)})\wh\bfV, -\iu \wh\bfU^* \bfnabla_{\bar  \bfZ} E(\wh\bfZ(t),\overline {\wh\bfZ(t)})\wh\bfV \rangle 
 \\
 &= 2 \,\Re \,(-\iu)\|\wh\bfU^* \bfnabla_{\bar  \bfZ} E(\wh\bfZ(t),\overline {\wh\bfZ(t)})\wh\bfV\|^2 = 0.
\end{align*}
Since $\wh\bfZ(t_0)=\bfZ_0$ 
by Lemma~\ref{lem:init}, we obtain for $\wh\bfZ_1=\wh \bfQ_1 + \iu \wh\bfP_1 =\wh\bfZ(t_1)$ that
$$
H(\wh\bfQ_1,\wh\bfP_1) =
\tfrac12 E(\wh\bfZ_1,\overline{\wh\bfZ_1}) = \tfrac12 E(\bfZ_0,\overline{\bfZ_0}) = H(\bfQ_0,\bfP_0).
$$
Since the truncation is such that $\| (\bfQ_1,\bfP_1) - (\wh\bfQ_1,\wh\bfP_1) \| =
\|\bfZ_1 - \wh\bfZ_1\| \le \vartheta$, we obtain 
\begin{align*}
\bigl| H(\bfQ_1,\bfP_1) - H(\bfQ_0,\bfP_0)\bigr| &= \bigl| H(\bfQ_1,\bfP_1) - H(\wh\bfQ_1,\wh\bfP_1)\bigr|
 \le \beta \vartheta
\end{align*}
with
$\beta= \max_{0\le\tau\le 1} \| \bfnabla H(\tau \bfQ_1 + (1-\tau)\wh\bfQ_1,\tau \bfP_1 + (1-\tau)\wh\bfP_1) \|$.
\qed
\end{proof}

	
	\section{A rank-adaptive robust low-rank Tucker tensor integrator}
	
     The solution $A(t)\in \R^{n_1\times\dots\times n_d}$ of a tensor differential equation
	\begin{equation} \label{eq:fullEq-ten}
	\dt{A}(t) = F(t, A(t)), 
	\qquad
	A(0) = A_0 
	\end{equation}	
	is approximated by evolving tensors $Y(t)\in \R^{n_1\times\dots\times n_d}$ of varying multilinear rank $\bfr=(r_1,\dots,r_d)$.
		Such tensors are represented in the Tucker form \cite{DeLauthawer:HOSVD} and are written in a notation following \cite{KoldaBader:TensorDec}:
	\begin{align} \label{Tucker}
	&Y(t) = C(t) \bigtimes_{i=1}^d \bfU_i(t) , 
	\\ 
	& \text{i.e.,}\quad y_{i_1,\dots,i_d}(t) = \sum_{j_1,\dots,j_d} c_{j_1,\dots,j_d}(t) \,u_{i_1,j_1}(t)\dots u_{i_d,j_d}(t),
	\nonumber
	\end{align}
	where the slim basis matrices $\bfU_i \in \mathbb{R}^{n_i \times r_i}$ have orthonormal columns and
	the smaller core tensor $C(t) \in \mathbb{R}^{r_1 \times \dots \times r_d }$ is of full multilinear rank $\bfr$.

		We present a rank-adaptive modification of the fixed-rank Tucker tensor integrator of \cite{CeL21} that retains its favourable properties. The integrator computes approximations $Y^n=C^n\bigtimes_{i=1}^d \bfU_i^n \approx  A(t_n)$ of an adaptively determined rank $\bfr^n=(r_1^n,\dots,r_d^n)$ at discrete times $t_n$ ($n=0,1,2,\dots$).

	\subsection{Formulation of the algorithm}	
		One time step of integration from time $t_0$ to $t_1=t_0+h$  starting from a  Tucker tensor 
		of multilinear rank $\bfr^0=(r_1^0,\dots,r_d^0)$ in factorized form, 
		$Y^0 = C^0 \bigtimes_{i=1}^d \bfU_i^0$,
	 computes an updated Tucker tensor 
		of multilinear rank $\bfr^1=(r_1^1,\dots,r_d^1)$ in factorized form, $Y^1 = C^1 \bigtimes_{i=1}^d \bfU_i^1$.
		In the following algorithm we let $\bfr=\bfr^0$ and we put a hat on quantities related to rank $2\bfr$.

		\begin{enumerate}
		\item Compute augmented basis matrices $\wh\bfU_i\in\R^{n_i\times 2r_i}$ for $i=1,\dots,d$ (in parallel):\\[2mm]
		Perform a QR factorization of the transposed $i$-mode matricization of the core tensor:
		$$ \text{\textbf{Mat}}_i(C^0)^\top = \textbf{W}_i \bfS_i^{0,\top} .$$
		With 
		$ \bfV_i^{0,\top} = 
		\textbf{W}_i^{\top} \bigotimes_{j \neq i}^d \bfU_j^{0,\top} \in \R^{r_i \times n_{\neg i}} 
		$
		(which yields 
		$ {\textbf{Mat}}_i(Y_0) = \bfU_i^0 \bfS_i^0 \bfV_i^{0, \top} $)\\[2mm]
		and the  matrix function $\bfF_{i}(t, \cdot) := \text{\textbf{Mat}}_i \circ F(t, \cdot) \circ  \textit{Ten}_i$,
		integrate from $t=t_0$ to $t_1$ the $n_i \times r_i$ matrix differential equation
		$$ 
		\dt{\bfK}_i(t) = \bfF_{i}(t,\bfK_i(t) \bfV_i^{0,\top}) \bfV_i^0,
		\qquad \bfK_i(t_0) = \bfU_i^0 \bfS_i^0.
		$$
				Determine the columns of $\wh \bfU_i\in \R^{n_i\times 2r_i}$ as an orthonormal basis of the range 
		of the $n_i\times 2r_i$ matrix $(\bfK_i(t_1),\bfU_i^0)$ (e.g.~by QR decomposition) and compute the $2r_i\times r_i$ matrix $\wh\bfM_i= \wh\bfU_i^{\top} \bfU_i^0$.

		\item Augment and update the core tensor $C^0\to \wh C(t_1)$:\\[2mm]
		Integrate from $t=t_0$ to $t_1$ the $2r_1 \times\dots\times 2r_d$  tensor differential equation
			\begin{align*}
			&\dt{\wh C}(t) = F \left( t, \wh C(t) \bigtimes_{i=1}^d \wh\bfU_i \right) \bigtimes_{i=1}^d \wh\bfU_i^{\top},
			\quad \wh C(t_0)  = C^0 \bigtimes_{i=1}^d \wh\bfM_i .
			\end{align*}
	\item Truncate to the tolerance~$\vartheta$ (cf.~\cite{de2000best}): Set $C_0=\wh C(t_1)$. For $i=1,\dots,d$ (sequentially), set $\wh C_i = C_{i-1}$,
		compute the SVD 
		$$ {\textbf{Mat}}_i(\wh C_i)= \wh \bfP_i \wh \bfSigma_i \wh \bfQ_i^\top$$ 
		and choose the new rank $r_i^1\le 2r_i$ as the 
		minimal number $r_i^1$ such that 
		$$
		\biggl(\ \sum_{j=r_i^1+1}^{2r_i} \sigma_j^2 \biggr)^{1/2} \le \vartheta/d.
		$$
		Let $\bfSigma_i$ be the  $r_i^1\times r_i^1$ diagonal matrix with the $r_i^1$ largest singular values of $\wh\bfSigma_i$ and let $\bfP_i^1\in \R^{2r_i\times r_i^1}$ and $\bfQ_i^1$ contain the first $r_i^1$ columns of $\wh \bfP_i$ and $\wh \bfQ_i$, respectively.\\
		Tensorize $C_i={\textit{Ten}}_i (\bfSigma_i^1\bfQ_i^{1,\top}) \in\R^{r_1^1\times\dots\times r_i^1\times 2r_{i+1}\times \dots\times 2r_d}$ and set $\bfU_i^1 = \wh \bfU \bfP_i^1\in \R^{n_i\times r_i^1}$.
	\end{enumerate} 
	With $C^1=C_d$, the approximation after one time step is then given by 
	\begin{equation}\label{Y1-ten}
	 Y^1 = C^1 \bigtimes_{i=1}^d \bfU_i^1.
	 \end{equation}
		To continue in time,  we take $Y^1$ as starting value for the next step and do another step of the integrator, and so on. 
 	
	\subsection{Properties}
\hskip 5mm -- The exactness property of Theorem~6 in \cite{CeL21} extends to the rank-adaptive Tucker tensor integrator, as can be shown by combining the proof of that theorem with the arguments in the proof of Theorem~\ref{thm:exact}. 

-- The robust error bound of Theorem~7 in \cite{CeL21} also extends to the rank-adaptive Tucker tensor integrator, with an extra term
$c_3\vartheta$ in the error bound as in Theorem~\ref{thm:robust}.

-- The preservation of (anti-)symmetry of Theorem~8 in \cite{CeL21} extends likewise.

-- So does the conservation of norm up to the truncation tolerance $\vartheta$ of Theorem~\ref{thm:norm}

-- and the decrease of the functional in gradient systems.

-- Also the near-conservation of energy for Schr\"odinger equations and Hamiltonian systems extends to the rank-adaptive Tucker tensor integrator.

The proofs of these extensions do not require new arguments beyond those of \cite{CeL21} and of Sections 2 and 3, therefore are omitted.

		\section{Numerical Experiments}
In this section, we present results of different numerical experiments.
These numerical simulations are implemented using \textsc{Matlab} R2019b and \textsc{Julia} 1.5.2.

\subsection{Error behaviour comparison}

 We compare the error behaviour of the ``unconventional'' matrix integrator of \cite{CeL21} with the rank-adaptive matrix integrator of Section \ref{subsec:alg-newint}. The matrix numerical example of \cite[Section 6.2]{CeL21} is considered:
\begin{equation*}
	\dt{\bfY}(t) = -\bfH[\bfY(t)], \quad \bfY(t_0) = \bfU_0 \bfS_0 \bfV_0^\top \in \R^{n \times n},
\end{equation*}
where
\begin{align*}
	&\bfH[\bfY] = \Big(\bfV_\text{cos} -\frac{1}{2}\bfD \Big) \ \bfY + \bfY \ \Big(\bfV_\text{cos} -\frac{1}{2}\bfD \Big)^\top \in \R^{n \times n},
	\\
	& \bfD = \texttt{tridiag}(-1,2,-1) \in \R^{n \times n} , 
	\\
	& \bfV_\text{cos} := \text{diag} \{ 1- \cos( \frac{2 \pi j }{n} ) \}, \quad j=-n/2, \dots, n/2-1 \ . 
\end{align*}


The diagonal matrix $\bfS_0 \in \R^{n \times n}$ has elements $ (S_0)_{ii} = 10^{-i} $ for $i=1, \dots n$ and the orthonormal matrices $\bfU_0, \bfV_0 \in \R^{n \times n}$ are randomly generated. 

The reference solution is computed with the \textsc{Matlab} solver \texttt{ode45} and  tolerance parameters \textsc{\{'RelTol', 1e-10, 'AbsTol', 1e-10\} }. The differential equations appearing in the substeps of the fixed-rank and adaptive-rank matrix integrators   are integrated with a second-order explicit Runge--Kutta method. 

\begin{figure}[t]
	\includegraphics[height=5cm, width=\textwidth]{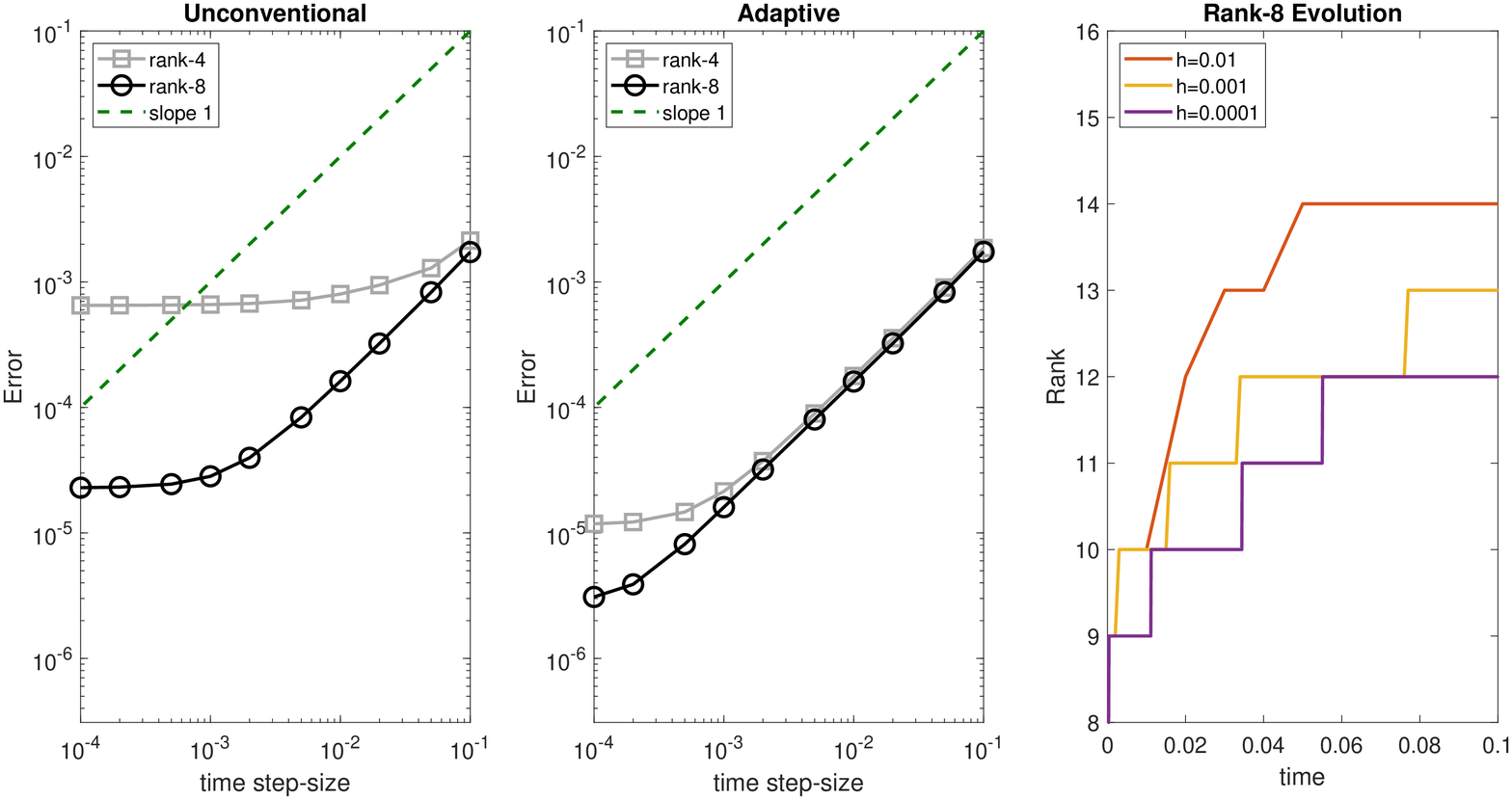}
	\caption{
		Comparison of the rank-$4$ and rank-$8$ approximations of the fixed-rank integrator of \cite{CeL21} and the new rank-adaptive integrator starting with initial ranks $4$ and $8$. The error is computed  with respect to the reference solution at time $T=0.1$. The rank evolution of the approximation with initial rank $8$ arising from the new matrix adaptive integrator for different time-step sizes is shown to the right.}
	\label{fig:err_adaptive}
\end{figure}

We choose $n=100$, ranks $r=4,8$ and final time $T=0.1$. The tolerance parameter selected for this numerical example is $\vartheta = 10^{-6}$.

The absolute errors $\| Y_n -A(t_n) \|_F$ at final time $t_n=T$ of the approximate solutions for different time-step sizes are shown in Figure~\ref{fig:err_adaptive}. The figure illustrates that the new rank-adaptive integrator retains first-order behavior in time and improves the error in the final approximation for smaller time-step sizes. The rank-adaptive integrator approximately  doubles the initial rank within this time interval.

\subsection{Radiation transport equation}
In the following numerical example, we consider a one-dimensional radiation transport equation. This equation is a mesoscopic model for the transport and interaction of radiation particles with a background material. For time $t \in [0, T]$ and particle density (or angular flux) $f=f(t,x,\mu)$, the radiation transport equation with isotropic scattering reads 
\begin{equation}
	\label{eq:rte}
	\begin{aligned}
		&\partial_t f + \mu \partial_x f + \sigma_s f = \frac{\sigma_s}{2} \int_{-1}^{1} f \,d\mu,
		\qquad 
		(x, \mu) \in [a,b] \times [-1,1] , \\[1mm]
		& f(t_0) = \frac{1}{\sqrt{2\pi}\sigma}\exp\Big(-\frac{x^2}{2\sigma^2} \Big).
	\end{aligned}
\end{equation}
The chosen initial condition is a Gaussian with constant deviation $\sigma = 3\cdot10^{-2}$. Hence, particles are initially positioned around $x=0$ and move into directions $\mu\in[-1,1]$. The reference solution to this problem is given by the \emph{standard de facto} Ganapol's benchmark test~\cite{ganapol2008analytical} and this problem has been investigated for dynamical low-rank approximations in \cite{PeMF20,PeM20}. As time increases, the scalar flux $\Phi(t,x) = \int_{-1}^{1} f(t,x,\mu) \,d\mu$ moves to the left and right side of the spatial domain, showing a discontinuous (or shock) profile at the front. When particles interact with the background material through collisions, which is the case for $\sigma_s>0$, the shock decreases over time and finally yields a smooth profile. In this work, we choose a scattering cross-section of $\sigma_s=1$.

The physical domain is discretized with a Lax--Friedrichs method in combination with a Legendre-polynomial expansion in the variable $\mu$.

A number of $N+1$ Legendre-polynomials is used and we split the space-interval $[a,b]$ in $N_x$ sub-intervals. Time-integration for the sub-steps of the adaptive integrator is performed with a first-order Runge-Kutta method and prescribed CFL number. 

The spatial domain has boundaries $a =-5$, $b=5$ and is discretized with $N_x=1000$ spatial cells. The polynomial representation of the scalar flux uses $N+1=200$ Legendre polynomials. A time step size is chosen with a CFL number of $0.99$. The tolerance parameter is set to $\vartheta = 10^{-1} \| \hat\bfSigma \|_2$ and $\vartheta = 5\cdot 10^{-2} \| \hat\bfSigma \|_2$. Here, the matrix $\hat\bfSigma$ arises from the SVD-factorization of the solution of the S-step computed with the adaptive matrix unconventional integrator, as illustrated in the last truncation step of the algorithm proposed in Section \ref{subsec:alg-newint}.

\begin{figure}
	\includegraphics[width=0.49\textwidth]{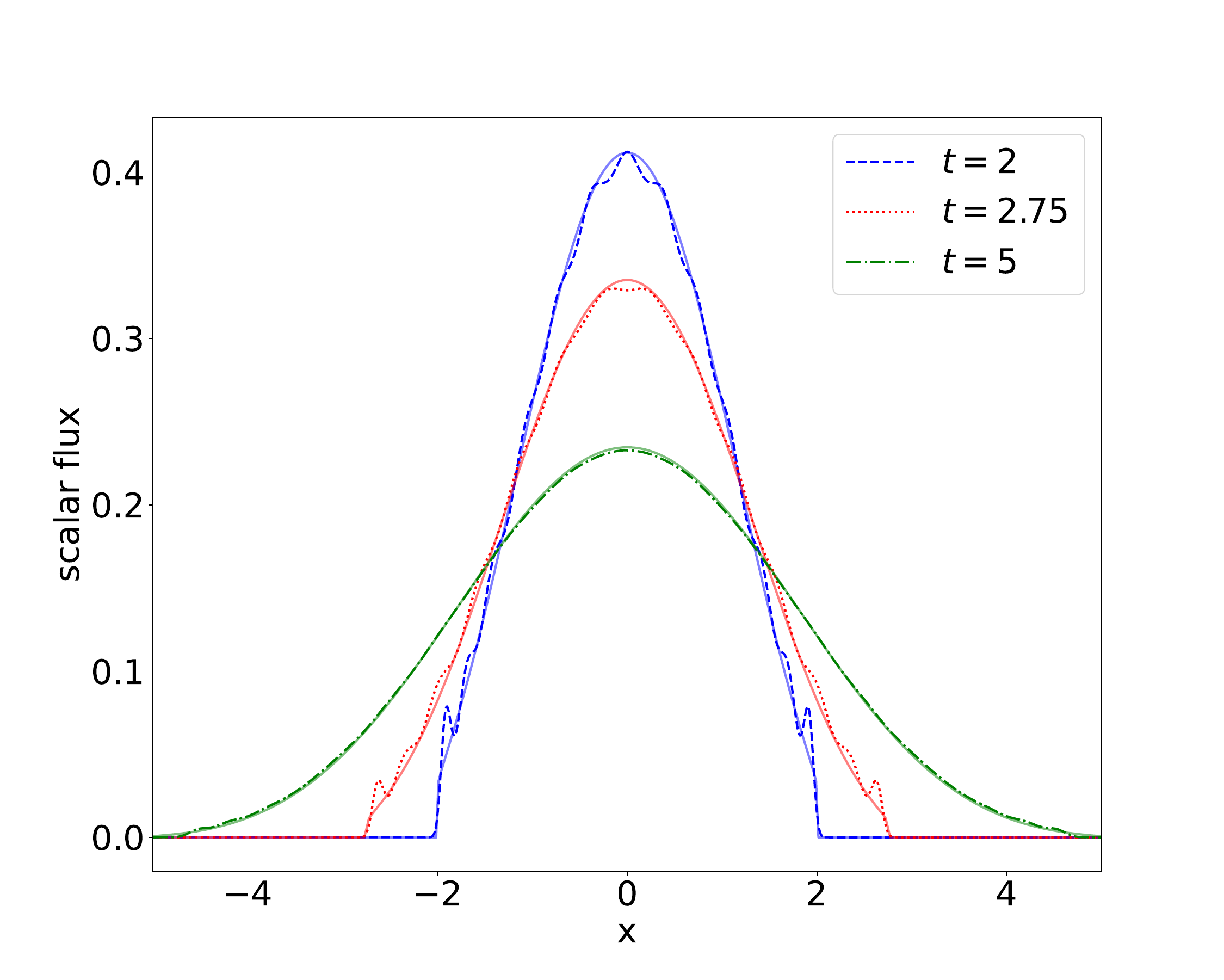}
	\includegraphics[width=0.49\textwidth]{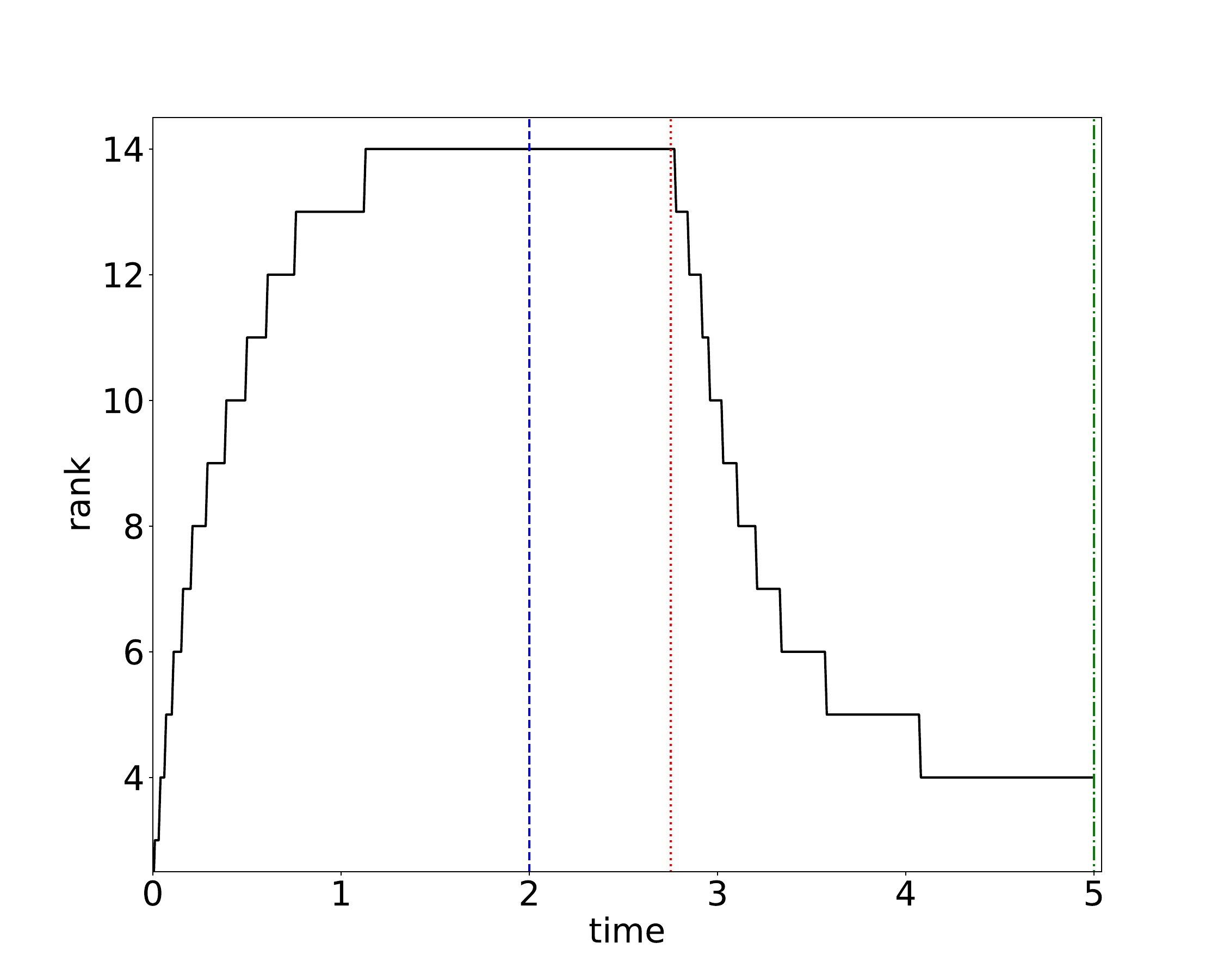}
		\includegraphics[width=0.49\textwidth]{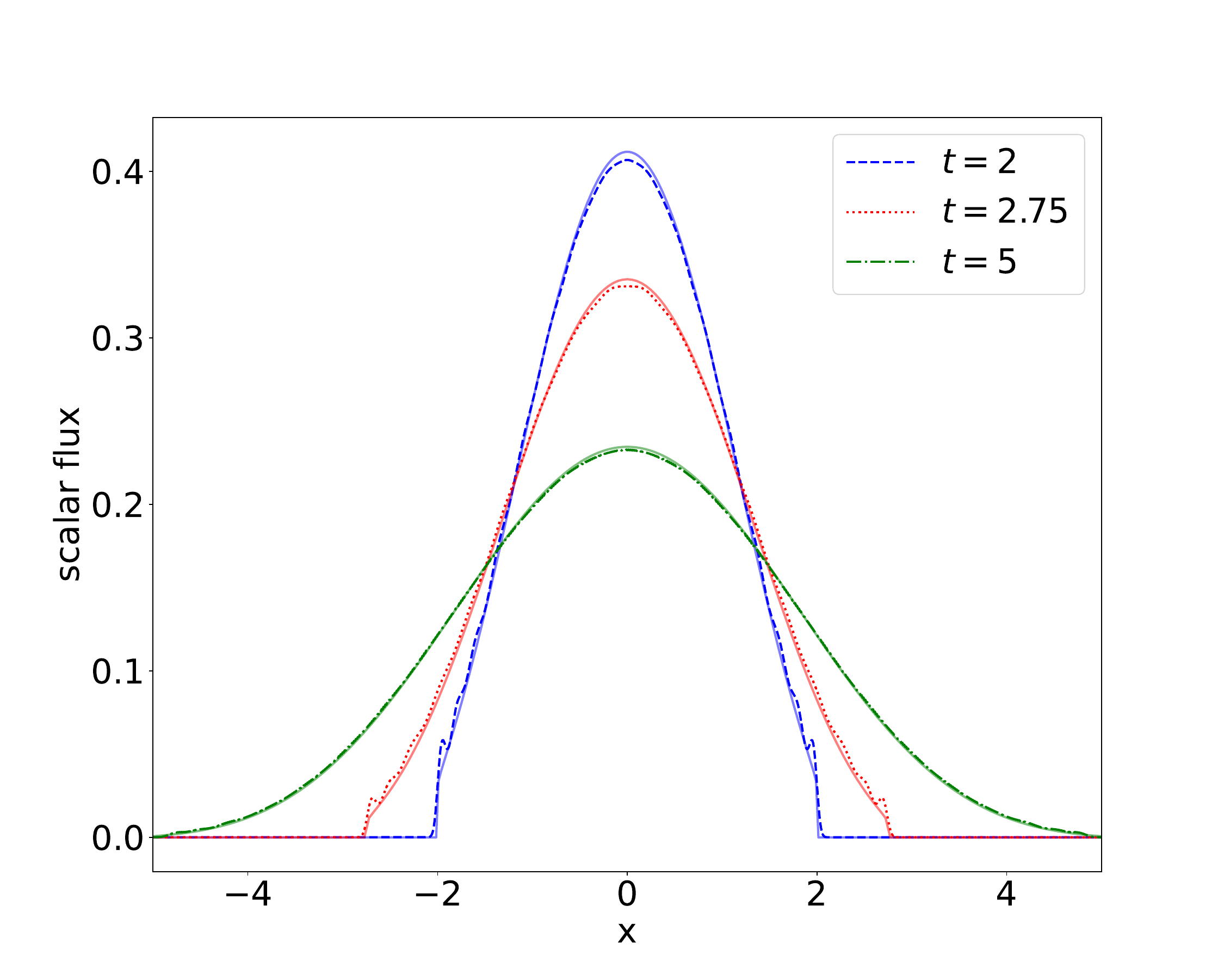}
	\includegraphics[width=0.49\textwidth]{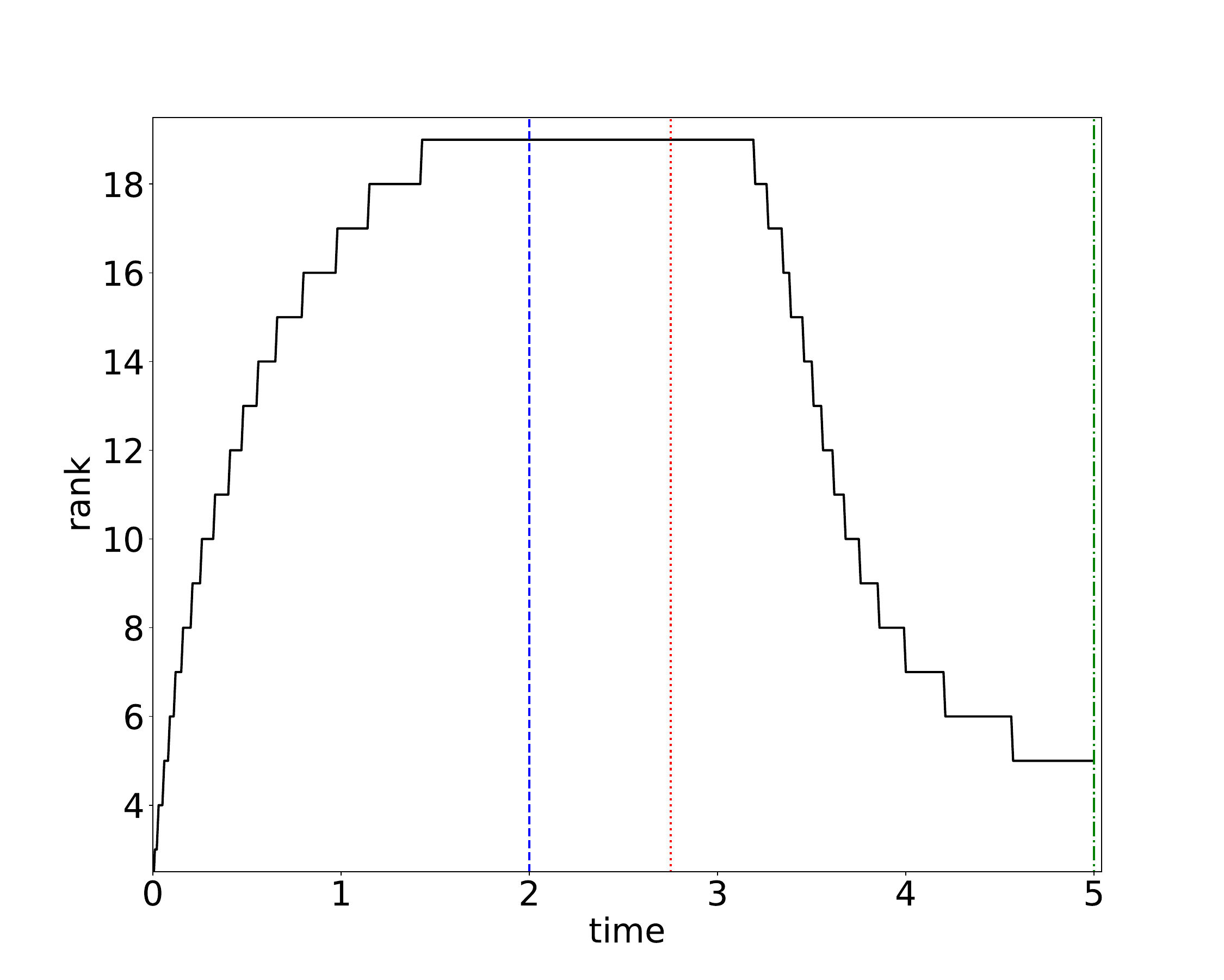}
	\caption{
		Top: Tolerance parameter $\vartheta = 10^{-1}\| \hat\bfSigma \|_2$. Bottom: Tolerance parameter $\vartheta = 5\cdot 10^{-2}\| \hat\bfSigma \|_2$. Left: Scalar flux of the reference solution (solid lines) of the radiation transport equation \eqref{eq:rte} in comparison with the approximation of the new rank-adaptive matrix integrator at different times. Right: Rank evolution of the approximation arising from the rank-adaptive integrator.}
	\label{fig:rte}
\end{figure}

In the left column of Figure \ref{fig:rte}, we consider the scalar flux of the 
reference solution at times $T\in\{2,2.75,5\}$ compared to the angular flux of a rank-$r$ approximation generated by 
the rank-adaptive matrix integrator of Section \ref{subsec:alg-newint}. In the right column, the rank evolution is plotted. The top row depicts the scalar flux and rank for a tolerance parameter $\vartheta = 10^{-1}\| \hat\bfSigma \|_2$ and the bottom row uses $\vartheta = 5\cdot 10^{-2}\| \hat\bfSigma \|_2$. Solid lines show the analytic solution computed according to \cite{ganapol2008analytical}. Our numerical experiment shows that decreasing the tolerance parameter improves the solution quality by increasing the chosen rank. However, the overall characteristics remain unchanged: While the rank of the initial condition is $r=1$, the adaptive algorithm increases the rank automatically. This increased rank is beneficial to capture the shock profile of the solution in the beginning. As the shock dissolves, the rank is reduced. We observe that the approximation arising from the adaptive method shows good agreement with the reference solution.

\subsection{Uncertainty quantification}

The following numerical experiment investigates Burgers' equation with uncertain initial condition
\begin{equation}
	\label{eq:burgers}
	\begin{cases}
		&\partial_t u(t,x,\bm{\xi}) + \partial_x \frac{u(t,x,\bm{\xi})^2}{2} = 
0,
		\qquad 
		(x, \bm{\xi}) \in [a,b] \times \Theta\;, \\
		& u(t_0,x,\bm{\xi}) = u_{\text{IC}}(x,\bm{\xi})\;.
	\end{cases}
\end{equation}
A uniformly distributed random vector $\bm{\xi}\in\Theta = [-1,1]\times[0,1]$ is used to model uncertainties in the initial condition. Since the 
dynamics of this non-linear hyperbolic model mimics advection effects that arise in gas dynamics, Burgers' equation is a standard model to test numerical methods. Numerical and analytic investigations of Burgers' equation with uncertain initial condition can for example be found in \cite{poette2009uncertainty,tryoen2010intrusive,despres2013robust}. To demonstrate 
the effects of varying smoothness in time, we choose an initial condition of
\begin{align*}
u_{\text{IC}}(x,\bm{\xi}) = 
\begin{cases} u_L, & \mbox{if } x< x_0+\sigma_1\xi_1 \\ u_L+\frac{u_R(\xi_2)-u_L}{x_0-x_1} (x_0+\sigma_1 \xi_1-x), & \mbox{if } x\in[x_0+\sigma_1 \xi_1,x_1+\sigma_1 \xi_1]\\
u_R(\xi_2), & \text{else }
\end{cases}\;,
\end{align*}
with $u_R(\xi_2)= u_R+\sigma_2\xi_2$. Note that this intial condition is similar to \cite{poette2009uncertainty,kusch2019maximum}, when additionally assuming an uncertain right state $u_R$ to increase computational complexity. At time $t=0$ the solution is a ramp or forming shock ranging 
from $x_0+\sigma_1\xi_1$ to $x_1+\sigma_1\xi_1$. Since the left state $u_L$ moves faster than the right state $u_R$ to the right side of the domain, a shock will form over time. The time at which the shock has fully developed depends on $\xi_2$ and is given by $t_{\text{s}} = \frac{x_1-x_0}{u_L-u_R(\xi_2)}$. We use the following parameter values:
\begin{center}
    \begin{tabular}{ | l | p{5.25cm} |}
    \hline
    $[a,b]=[0,1]$ & range of spatial domain \\
    $T = 0.04$ & end time \\
    $N_x=600$ & number of spatial cells \\
    $x_0 = 0.3, x_1=0.4, u_L = 12, u_R = 1$ & parameters of initial condition \\
    $\sigma_1 = 0.2, \sigma_2 = 5$ & parameters of the uncertainty \\
    $\vartheta = \{1,1.2,1.5\}\cdot 10^{-2}\| \hat\bfSigma \|_2$ & tolerance parameter \\
    \hline
    \end{tabular}
\end{center}
A spatial discretization of equation \eqref{eq:burgers} is performed by a first order finite volume method with Lax--Friedrichs numerical fluxes. As time discretization an explicit Euler method is chosen. The stabilization of the finite volume method is applied in the $K$, $L$ and $S$-steps. We do not split the uncertain domain and choose a modal representation of the uncertain basis making use of tensorized Legendre polynomials. For each 
uncertainty, Legendre polynomials up to degree $19$, i.e., $20^2 = 
400$ polynomials, are used to represent the uncertain basis.

Numerical results for this testcase are presented in Figure~\ref{fig:ExpectationStdU}. An analytic solution for given values of $\bm{\xi}$ is determined with characteristics. Expectation and variance are computed by using a fine tensorized quadrature rule with $200\cdot 200$ Gauss-Legendre quadrature points. The resulting expectation is depicted in red and the corresponding standard deviation is shown in blue. The rank-adaptive method proposed in Section~\ref{subsec:alg-newint} shows satisfactory agreement with the analytic solution, especially for the expectation. Initially, the rank is chosen as $40$ and the method reduces this rank after the first time step. As time increases, a shock (or discontinuity) forms in both, the spatial and uncertain domain. The adaptive method captures the growing solution complexity by increasing the rank. After a certain time, the method remains at a fixed higher rank for all chosen tolerance parameters, where the rank depends on the chosen tolerance. Note that this rank will be reached after the shock has fully developed. This is most likely due to the sharpening of the numerical solution which 
results from increasing the rank. As a result, the singular values of the $\bfS$ 
matrix continue to grow. 

To point out differences to the fixed-rank integrator, we include a comparison of the rank-adaptive integrator with tolerance parameter $\vartheta = 0.015$ with numerical solutions for fixed rank $9$ and $25$. These ranks are the minimal and maximal rank chosen by the rank-adaptive integrator during the computation. The numerical solution of the rank-adaptive algorithm shows good agreement with the fixed-rank integrator when using a constant rank of $25$. In comparison to a fixed rank of $9$, the adaptive method yields a strongly improved solution quality.

\begin{figure}[h!]
\centering
\includegraphics[width=0.49\textwidth]{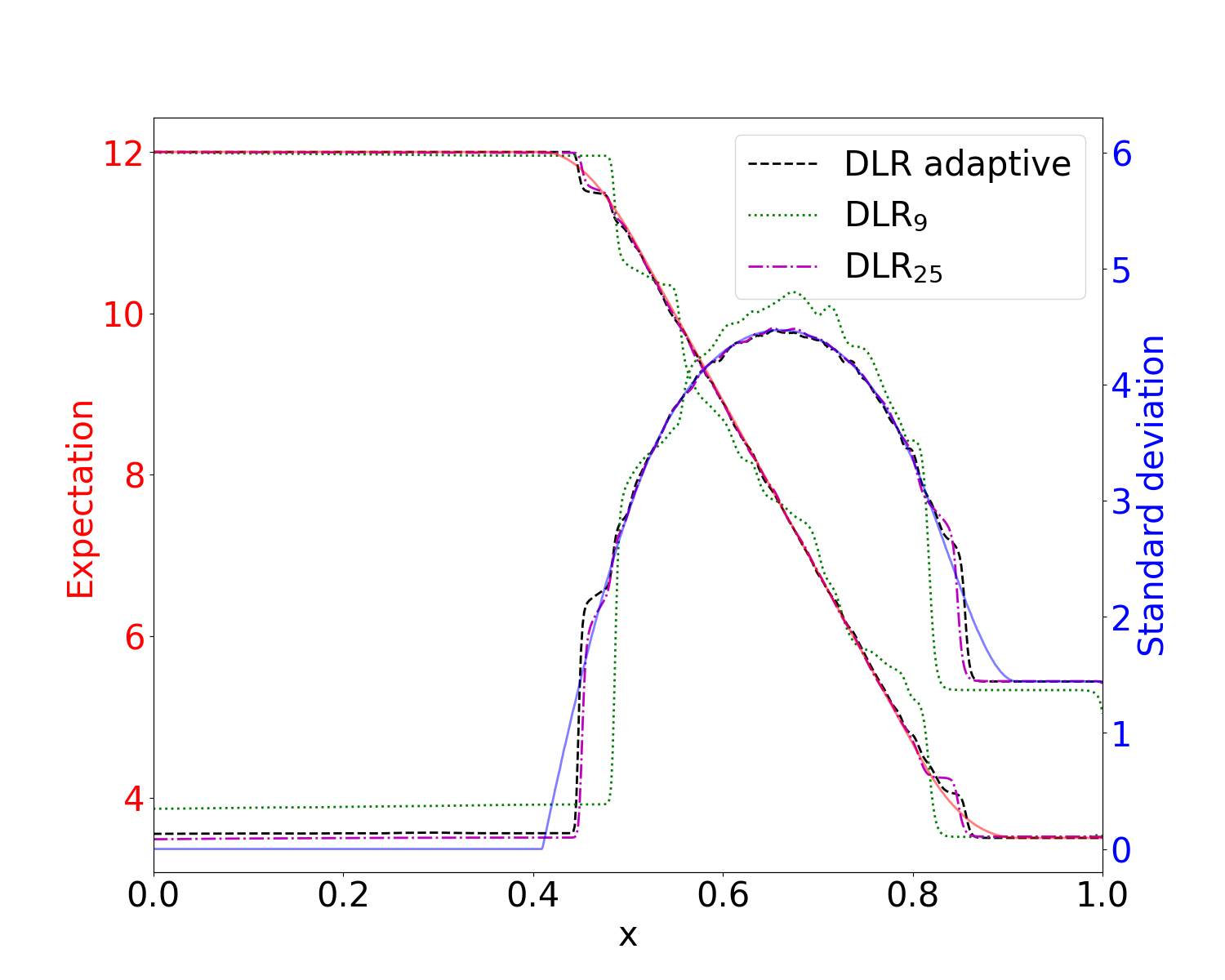}
\includegraphics[width=0.49\textwidth]{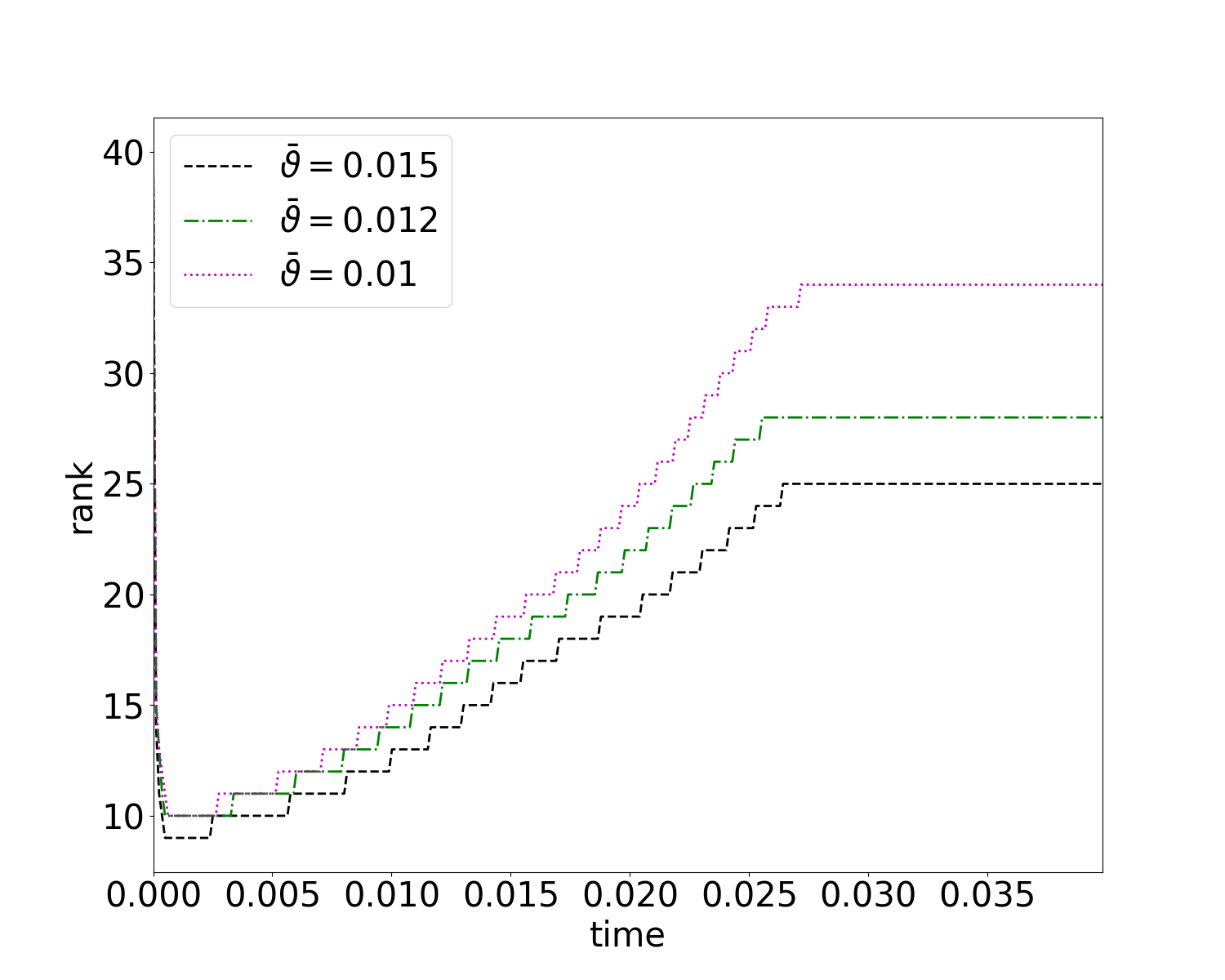}
\includegraphics[width=1.0\textwidth]{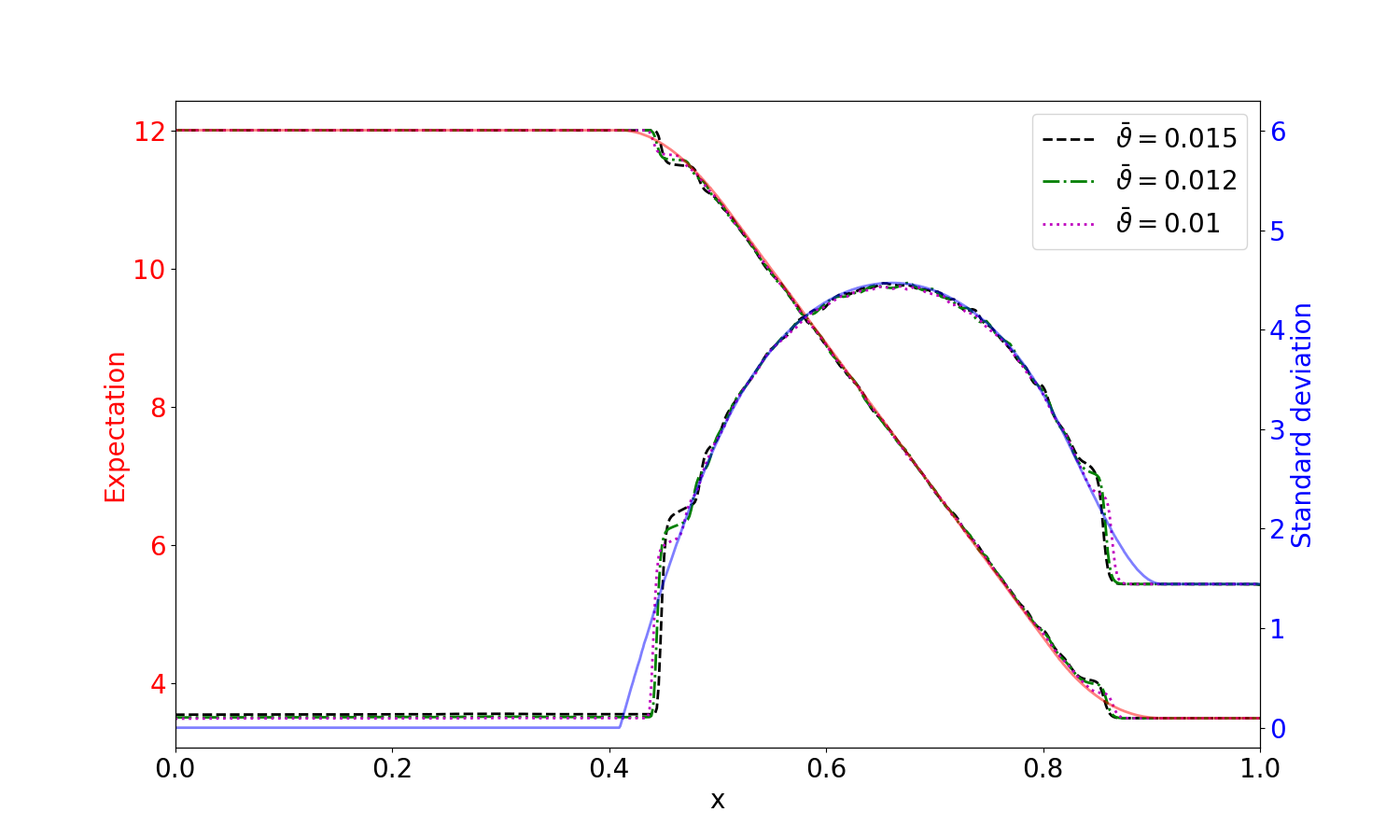}
\caption{Top left: Expected value (red) and variance (blue) of the analytic solution compared between the rank-adaptive and fixed-rank integrators. The rank-adaptive integrator uses the truncation tolerance $\vartheta = \bar{\vartheta}\| \hat\bfSigma \|_2$ where $\bar{\vartheta} = 0.015$. The fixed-rank integrator uses $r=9$ and $r=25$. Top right: Rank evolution of the rank-adaptive integrator over time using $\bar{\vartheta}\in\{0.01,0.012,0.015\}$. Bottom: Corresponding approximations for expectation and standard deviation.}
\label{fig:ExpectationStdU}
\end{figure}

	\begin{acknowledgements} 
		This work was funded by the Deutsche Forschungsgemeinschaft (DFG, German Research Foundation) --- Project-ID 258734477 --- SFB 1173. It started out from a discussion at the 2021 annual meeting of SFB 1173. We thank Martin Frank (KIT) for fruitful discussions on radiation transport theory.
	\end{acknowledgements}
	
	\bibliographystyle{abbrv}
	\bibliography{dlradapt} 
	
\end{document}